\documentclass{amsart}
\usepackage[paperwidth=165mm,paperheight=241mm,hmargin=2cm,vmargin=2cm,centering]{geometry}
\usepackage{amssymb,amsmath,amsthm,amsfonts,amscd}
\usepackage{graphicx,xcolor}
\usepackage{mathrsfs}
\usepackage{hyperref}
\hypersetup{colorlinks=true,allcolors=[rgb]{0,0,0.6}}
\usepackage[utf8]{inputenc}
\usepackage[english]{babel}
\usepackage{dsfont}
\usepackage[longnamesfirst,numbers,sort&compress]{natbib}
\usepackage[shadow,textwidth=12mm,textsize=tiny,obeyDraft]{todonotes}
\usepackage{enumerate}
\usepackage{setspace}
\DeclareMathOperator{\Tr}{Tr}

\renewcommand{\Im}{\mathrm{Im}}
\renewcommand{\Re}{\mathrm{Re}}
          \newtheorem{thm}{Theorem}[section]
          \newtheorem{proposition}[thm]{Proposition}
          \newtheorem*{proposition*}{Proposition}
          \newtheorem{lemma}[thm]{Lemma}
          
          \newtheorem{definition}[thm]{Definition}
          
          \newtheorem{hyp}[thm]{Hypothesis}
          \theoremstyle{definition}

          \newtheorem{remark}[thm]{Remark}
          \newtheorem*{remark*}{Remark}

\renewcommand{\setminus}{\smallsetminus}
\begin{document}
\bibliographystyle{abbrvnat} \title{Concentration of cylindrical Wigner
  measures 
}\author{Marco Falconi}\address{I-Math, Universität Zürich; Winterthurerstrasse 190, CH-8057 Zürich}
\email{marco.falconi@math.uzh.ch} 
\urladdr{http://user.math.uzh.ch/falconi/}\thanks{The author acknowledges the support of MIUR through the FIR
  grant 2013 \href{http://www.cond-math.it}{Condensed Matter in Mathematical Physics (Cond-Math)} (code
  RBFR13WAET). The final version of this preprint is to appear in
  \href{http://www.worldscientific.com/worldscinet/ccm}{Communications in Contemporary Mathematics} $^{\text{\copyright}}$ 2017
  World Scientific Publishing Company}

\date{\today}
\begin{abstract}
  In this note we aim to characterize the cylindrical Wigner measures associated to regular quantum states in
  the Weyl C*-algebra of canonical commutation relations. In particular, we provide conditions at the quantum
  level sufficient to prove the concentration of all the corresponding cylindrical Wigner measures as Radon
  measures on suitable topological vector spaces. The analysis is motivated by variational and dynamical
  problems in the semiclassical study of bosonic quantum field theories.
\end{abstract}
\maketitle

\onehalfspacing

\section{Introduction}
\label{sec:introduction}

In this brief note, we discuss the concentration of cylindrical Wigner measures. Cylindrical Wigner measures
on a topological vector space $V$ have been introduced in
\citep{Falconi:2016ab
} as the classical limit points, in a suitable topology, of regular quantum states in the Weyl C*-algebra of
canonical commutation relations $\mathbf{W}_h(V',\sigma)$ -- where $V'$ is the continuous dual of some
topological vector space $V$. The dual $V'$ is endowed with a symplectic form, and it is interpreted as the
classical phase space. Any cylindrical measure on a (infinite dimensional) topological vector space $V$ is a
Borel Radon measure on a ``bigger'' space $\overline{V}$ (see Section \ref{sec:conc-wign-meas}, or
\citep{MR0426084
}, for additional details). However, the space $\overline{V}$ is usually too large, in the sense that it is
not manageable to study dynamical or variational problems in the semiclassical or mean field analysis of
bosonic many-particle systems and fields. In applications \citep[see e.g.][and references thereof
contained]{Ammari:2017aa
  ,ammari:15
  ,Ammari:2014ab
  ,Ammari:2014aa
  ,2011arXiv1111.5918A
  ,MR2802894
  ,Correggi:2017aa
  ,Lewin:2014aa
  ,Lewin:2013aa
  ,MR3390788
} the attention is often restricted, due to the physical properties of the system, to quantum states whose
corresponding Wigner measures are concentrated as Radon measures on some manageable (\emph{e.g.} separable
Hilbert) topological vector space. It is therefore natural to ask for conditions on regular quantum states
that are sufficient to yield concentration, on some given space, of all the corresponding Wigner measures.

In Section \ref{sec:separ-hilb-spac}, we provide sufficient conditions to prove concentration in separable
complex Hilbert spaces (with inner product compatible with the symplectic form $\sigma$); in Section
\ref{sec:dual-nuclear-space} we provide sufficient conditions to prove concentration on continuous duals of
nuclear spaces (with any locally convex topology between the ultraweak and the Mackey topology); in
Section~\ref{sec:locally-conv-spac} we provide sufficient conditions to prove concentration in duals of Banach
spaces, endowed with the ultraweak topology. For the Weyl C*-algebra
$\mathbf{W}_h(\mathbb{F}\mathscr{H}, \Im \langle \cdot , \cdot \rangle_{\mathscr{H}})$ -- built on the
classical separable Hilbert phase space $\mathbb{F}\mathscr{H}$ associated to a complex Hilbert space
$\mathscr{H}$ -- sufficient conditions on Fock-normal quantum states to prove concentration on $\mathscr{H}$
(and on suitable Hilbert spaces included in $\mathscr{H}$) have been given in
\citep{ammari:nier:2008
  ,2011arXiv1111.5918A
}. Theorem~\ref{thm:1} can be seen as a generalization of the aforementioned results, to more general regular
quantum states on arbitrary Weyl C*-algebras. Such generalization is motivated by concrete problems in the
variational and dynamical semiclassical analysis of bosonic quantum field theories, as explained in
Sections~\ref{sec:motivation} and~\ref{sec:some-expl-exampl}. Our analysis is based on the ``cylindrical''
(finite-dimensional) pseudodifferential calculus for bosonic quantum field theories, that is outlined in
Section~\ref{sec:cylindr-quant-regul}.

\begin{remark*}
  Theorems \ref{thm:1}, \ref{thm:2}, and \ref{thm:3} can be easily adapted to regular quantum states in the
  tensor product of the C*-algebra of canonical commutation relations $\mathbf{W}_h(V',\sigma)$ with another
  C*-algebra $\mathfrak{A}$ describing additional degrees of freedom of the physical system
  \citep[see][]{Falconi:2016ab
  }. The only difference is that the Wigner measures are, in this case, vector-valued and take values in the
  space $\mathfrak{A}'_+$ of (positive) states acting on $\mathfrak{A}$.
\end{remark*}

\subsection{Motivation}
\label{sec:motivation}

This technical note is motivated by some problems we encountered in the study of the semiclassical properties
of bosonic quantum systems \citep{Ammari:2014aa
  ,Correggi:2017aa
}. Let us discuss briefly a concrete case. Let $d\geq 2$, and consider a self-adjoint and bounded from below
operator $H_h$ on the symmetric Fock space
\begin{equation*}
  \Gamma_{\mathrm{s}}\bigl(L^2 (\mathds{R}^d )\bigr)=\bigoplus_{n\in \mathds{N}}L^2_n =\mathds{C}\bigoplus_{n\in \mathds{N}_{*}}\underbrace{L^2\otimes_{\mathrm{s}}\dotsm\otimes_{\mathrm{s}}L^2}_n\; .
\end{equation*}
$H_h$ is the quantization of some classical symbol $\mathscr{E}$, and it is defined on the domain
$D\bigl(\mathrm{d}\Gamma_h(\lvert k \rvert_{}^{})\bigr)$, where
\begin{equation*}
  L^2_n\ni  \bigl( \mathrm{d}\Gamma_h(\lvert k \rvert_{}^{})\psi_{h}\bigr)_n(k_1,\dotsc,k_n)= h\sum_{j=1}^n\lvert k_j  \rvert_{}^{}\psi_{h,n}(k_1,\dotsc,k_n)\; . 
\end{equation*}
Throughout the paper, $h\in (0,1)$ denotes the semiclassical parameter. The second quantized operator
$\mathrm{d}\Gamma_h(\lvert k \rvert_{}^{})$ is the free energy of \emph{massless} bosonic scalar fields,
\emph{e.g.} radiation. Suppose now that we want to investigate the convergence, as $h\to 0$, of the ground
state energy $\underline{\sigma}(H_h)$. Let $j\in \mathds{N}$, and
$\psi_h^{(j)}\in D\bigl(\mathrm{d}\Gamma_h(\lvert k \rvert_{}^{})\bigr)$ be a vector in a normalized
minimizing sequence of $H_h$. Since $(\psi_h^{(j)})_{h\in (0,1)}$ is norm-bounded uniformly with respect to
$h$, its cluster points are cylindrical Wigner measures
\citep[see][]{Falconi:2016ab
}. If all the cluster points concentrate as measures in the classical energy space
$\mathscr{F}\dot{H}^{1/2}(\mathds{R}^d)= L^2 (\mathds{R}^d, \lvert k \rvert_{}^{}\mathrm{d}k )$, we can use
the information to prove the convergence
\begin{equation*}
  \lim_{h\to 0}\underline{\sigma}(H_h)=\inf_{\alpha\in \mathscr{F}\dot{H}^{1/2}}\mathscr{E}(\alpha)\; ;
\end{equation*}
see \citep{Ammari:2014aa
} for a more detailed explanation of the strategy. It is well-known
\citep{ammari:nier:2008
} that
\begin{align}
  \label{eq:25}
  \bigl\langle \varphi_h  &,N_h \varphi_h \bigr\rangle_{\Gamma_{\mathrm{s}}}=\bigl\langle \varphi_h  , \mathrm{d}\Gamma_h(1) \varphi_h \bigr\rangle_{\Gamma_{\mathrm{s}}}\leq \mathcal{K}
\end{align}
is sufficient to prove concentration of the cluster points of $(\psi_h)_{h\in (0,1)}$ in
$L^2 (\mathds{R}^d )$. However, for $(\psi_h^{(j)})_{h\in (0,1)}$ we have the sole information
\begin{equation}
  \label{eq:24}
  \bigl\langle \psi^{(j)}_h  , \mathrm{d}\Gamma_h(\lvert k  \rvert_{}^{}) \psi^{(j)}_h \bigr\rangle_{\Gamma_{\mathrm{s}}}\leq \mathcal{K}\; .
\end{equation}
In addition, since the symmetric difference $L^2\ominus \mathscr{F}\dot{H}^{1/2}$ is not empty, there are families of
quantum states $(\varphi_h)_{h\in (0,1)}$ that converge to Wigner measures that are Radon on
$\mathscr{F}\dot{H}^{1/2}$ and concentrated outside of $L^2$. Therefore it is not possible to restrict the
analysis to vectors that satisfy both \eqref{eq:25} and \eqref{eq:24}; and in fact the miminimizing sequence
might not satisfy \eqref{eq:25}. As we will prove in Theorem~\ref{thm:1}, bound~\eqref{eq:24} is sufficient to
have concentration of all the corresponding cylindrical Wigner measures in $\mathscr{F}\dot{H}^{1/2}$. Let us
formulate here the result explicitly, and give an outline of the proof. This would serve to explain the key
idea behind the proof of the more general Theorem~\ref{thm:1}.
\begin{proposition*}
  Let $(\varphi_h)_{h\in (0,1)}\subset \Gamma_{\mathrm{s}}\bigl(L^2 (\mathds{R}^d )\bigr)$. Suppose there exists $\mathcal{K}>0$ such that
  \begin{equation*}
    \bigl\langle \varphi_h  , \mathrm{d}\Gamma_h(\lvert k  \rvert_{}^{}) \varphi_h \bigr\rangle_{\Gamma_{\mathrm{s}}}\leq \mathcal{K}\; .
  \end{equation*}
  Then any Wigner measure that is a cluster point of $(\varphi_h)_{h\in (0,1)}$ is a Borel Radon measure on
  $\mathscr{F}\dot{H}^{1/2}(\mathds{R}^d)$.
\end{proposition*}
\begin{proof}[Sketch of the proof]
  Let $G\subset L^2$ be a finite dimensional subspace with orthonormal basis $\{g_j\}_{j\in \{0,\dotsc,n\}}$; there is a
  canonical identification of the Fock-space creation and annihilation operators
  $\{a^{\#}_h(g_j)\}_{j\in \{0,\dotsc,n\}}$ with the standard creation and annihilation operators associated
  to the Heisenberg group \citep[see
  e.g.][]{ammari:nier:2008
  }
  \begin{equation*}
    \mathbf{H}(\mathbb{F}G, \Im \langle \cdot , \cdot\rangle_2) =\mathbf{H}\bigl(\mathrm{span}_{\mathds{R}}(g_1, i g_{1},\dotsc,g_n,i g_n),\Im \langle \cdot , \cdot\rangle_2\bigr)\; .
  \end{equation*}
  Since we want to prove concentration of measures on a different space, namely $\mathscr{F}\dot{H}^{1/2}$, we
  need to make an alternative identification. This is the key step in the proof, hence let us explain in
  detail how to make the identification.

  Let $(e_j)_{j\in \mathds{N}}\subset \mathscr{S}(\mathds{R}^d)$ be an orthonormal basis of the Hilbert space
  $\mathscr{F}\dot{H}^{1/2}(\mathds{R}^d)$. Then
  \begin{equation*}
    E_R=\mathrm{span}_{\mathds{C}}\{e_0,e_1,\dotsc,e_R\}\; ,\; R\in \mathds{N}
  \end{equation*}
  is an increasing sequence of finite dimensional complex Hilbert spaces ordered by inclusion, such that
  $\bigcup_{R\in \mathds{N}}E_R=\mathscr{F}\dot{H}^{1/2}$. Equivalently, each $E_R$ can be seen as a real
  symplectic space $\mathbb{F}E_R$ of dimension $2R$ with symplectic form
  $\sigma_E(\,\cdot \,,\,\cdot \,)=\Im \, \langle \,\cdot\, , \,\cdot\,
  \rangle_{\mathscr{F}\dot{H}^{1/2}}$. Now let
  $(\varepsilon_j)_{j\in \mathds{N}}\subset \mathscr{S}(\mathds{R}^d)$ be defined by $\varepsilon_j(k)=\lvert k \rvert_{}^{1/2}e_j(k)$, and
  \begin{equation*}
    \mathcal{E}_R=\mathrm{span}_{\mathds{C}}\{\varepsilon_0,\varepsilon_1,\dotsc,\varepsilon_R\}\; ,\; \sigma_{\mathcal{E}}(\,\cdot \,,\,\cdot \,)=\Im \langle\,\cdot \,   , \,\cdot \, \rangle_2\; .
  \end{equation*}
  Clearly, $(\mathbb{F}E_R,\sigma_E)$ and $(\mathbb{F}\mathcal{E}_R,\sigma_{\mathcal{E}})$ are the same symplectic spaces under the
  identification $e_j\mapsto \varepsilon_j$ (or equivalently $E_R$ and $\mathcal{E}_R$ are the same complex Hilbert spaces).

  Consider now the finite-dimensional Heisenberg group
  $\mathbf{H}(\mathbb{F}E_R,\sigma_E)$; since
  $E_R\subset L^2$, the Heisenberg group has a unitary representation in $\Gamma_{\mathrm{s}}(L^2)$, given by
  \begin{equation*}
    e^{it}W_h(\lvert\, \cdot \, \rvert_{}^{1/2}e(\cdot ))\; ,\; (e,t)\in E_R\times \mathds{R}\; ;
  \end{equation*}
  where the Weyl operators are customarily defined by means of the usual creation and annihilation operators
  \begin{equation*}
    W_h(f)=e^{i (a^{*}_h(f)+a_h(f))}\;,\quad \forall f,g\in L^2\; [a_h(f),a^{*}_h(g)]=h\langle f  , g \rangle_2\; .
  \end{equation*}
  On the other hand, $\mathrm{d}\Gamma_h(\lvert k \rvert_{}^{})$ is the Wick quantization of
  $\lVert z \rVert_{\mathscr{F}\dot{H}^{1/2}}^2$. It follows that
  \begin{equation*}
    \mathrm{d}\Gamma_h(\lvert k \rvert_{}^{})=\sum_{j\in \mathds{N}}^{}\mathbf{a}_h^{*}(\varepsilon_j)\mathbf{a}_h(\varepsilon_j)=\mathbf{N}_h\; ,
  \end{equation*}
  where the modified creation and annihilation operators $\mathbf{a}^{\#}_h(\cdot )$ are defined by
  \begin{align*}
    \bigl(\mathbf{a}_h(f)\psi_h\bigr)_n(k_1,\dotsc,k_n)&=\sqrt{h(n+1)}\int_{\mathds{R}^d}^{}\lvert k  \rvert_{}^{1/2} f(k)\psi_{h,n+1}(k,k_1,\dotsc,k_n)  \mathrm{d}k\;;\\
    \bigl(\mathbf{a}_h^{*}(f)\psi_h\bigr)_n(k_1,\dotsc,k_n)&=\sqrt{h/n}\sum_{i=1}^n\lvert k_i  \rvert_{}^{1/2}f(k_i)\psi_{h,n-1}(k_1,\dotsc,\underline{k}_j,\dotsc,k_n)\; ;
  \end{align*}
  where $\underline{k}$ means that the variable is missing. The restriction of
  $\lVert z \rVert_{\mathscr{F}\dot{H}^{1/2}}^2$ to $E_R$ takes the form
  \begin{equation*}
    \sum_{j=0}^R\biggl\lvert \int_{\mathds{R}^d}^{}\lvert k  \rvert_{}^{1/2}\bar{z}(k) \varepsilon_j(k)  \mathrm{d}k  \biggr\rvert_{}^2\; ,
  \end{equation*}
  and therefore its Wick quantization is
  \begin{equation*}
    \mathrm{d}\Gamma_h(\lvert k \rvert_{}^{})_R=\sum_{j=0}^R \mathbf{a}_h^{*}(\varepsilon_j)\mathbf{a}_h(\varepsilon_j)=\mathbf{N}_{h,R}\; .
  \end{equation*}
  In addition, by means of the modified creation and annihilation operators it is possible to write the Weyl
  operators $W_h(\lvert \,\cdot \,  \rvert_{}^{1/2}e(\cdot ))$, $e\in E_R$, as
  \begin{equation*}
    W_h(\lvert \,\cdot \,  \rvert_{}^{1/2}e(\cdot ))=e^{i (\mathbf{a}^{*}_h(e)+\mathbf{a}_h(e))}\; .
  \end{equation*}
  Therefore $\{\mathbf{a}^{\#}_h(e_j)\}_{j\in \{0,\dotsc,R\}}$ are identified with the usual creation and annihilation
  operators on the finite-dimensional Fock representation $\Gamma_{\mathrm{s}}(\mathds{C}^{R+1})$, and
  $\mathds{C}^{R+1}\cong E_R$ as complex Hilbert spaces.

  Let us remark again that the identifications above are the key ingredient that allows us to use the standard
  tools of pseudodifferential calculus associated to the Heisenberg group
  $\mathbf{H}(\mathbb{F}E_R,\sigma_E)$. The result is then proved combining such finite-dimensional
  pseudodifferential techniques with a Prokhorov-type tightness argument on the projective family of
  cylindrical measures, exploiting the fact that $\bigcup_{R\in \mathds{N}}E_R=\mathscr{F}\dot{H}^{1/2}$. The
  details can be found in the proof of Theorem~\ref{thm:1}.
\end{proof}

Other applications of Theorem~\ref{thm:1} that are important for the semiclassical analysis of physical
problems in bosonic quantum field theory will be given in Section~\ref{sec:some-expl-exampl}.

\section{Cylindrical semiclassical analysis}
\label{sec:cylindr-quant-regul}

Let $V$ be a real locally convex space. A cylindrical measure on $V$ is a projective family
$M=(\mu_{\Phi})_{\Phi\in F(V)}$ of finite Borel Radon measures on $V/\Phi$, where $F(V)$ is the set of
\emph{closed} subspaces of $V$ with finite codimension. To every $\Phi\in F(V)$, there corresponds the finite
dimensional polar (orthogonal) $\Phi^{\circ}\subset V'$, where $V'$ is the continuous dual of
$V$.

\begin{definition}[Cylindrical symbol]
  \label{def:1}
  A function $f:V\to \mathds{C}$ is a cylindrical symbol with base $\Phi\in F(V)$ iff there exists a function
  $f_{\Phi}:V/\Phi\to \mathds{C}$ such that
  \begin{equation*}
    \forall v\in V\,,\, f(v)=f_{\Phi}([v]_{\Phi})\; .
  \end{equation*}
\end{definition}

Let $f$ be a cylindrical symbol, $f_{\Phi}$ the corresponding finite dimensional function. If $f_{\Phi}$ is
measurable, it is possible to define its cylindrical integral with respect to any cylindrical measure $\mu$:
\begin{equation}
  \label{eq:2}
  \int_V^{\mathrm{(cyl)}}\mspace{-15mu}f(v)  \mathrm{d}M(v)=\int_{V/\Phi}^{}f_{\Phi}(w)  \mathrm{d}\mu_{\Phi}(w)\; .
\end{equation}
If we denote by $p_{\Phi}:V\to V/\Phi$ the canonical projection onto the equivalence classes,
$^{\mathrm{t}}p_{\Phi}:(V/\Phi)'\to \Phi^{\circ}$ is an isomorphism. It follows that for any $f_{\Phi}\in L^1 (V/\Phi )$,
\begin{equation}
  \label{eq:11}
  \int_V^{\mathrm{(cyl)}}\mspace{-15mu}f(v)  \mathrm{d}M(v)=\int_{V/\Phi}^{}f_{\Phi}(w)  \mathrm{d}\mu_{\Phi}(w)=\int_{\Phi^{\circ}}^{}\hat{f}_{\Phi} (\xi)\hat{M}(2\pi \xi)  \mathrm{d}L_{\Phi^{\circ}}(\xi)\; .
\end{equation}
Here $\hat{M}$ is the Fourier transform\footnote{We define the Fourier transform of a distribution both as an
  isometry on $L^2$ and an algebra homomorphism on $L^1$, \emph{i.e.}
  $\hat{f}(\xi)=\int_{}^{}e^{-2\pi i\xi\cdot x}f(x) \mathrm{d}x$. On the contrary we define, following common
  practice, the Fourier transform of a measure by $\hat{\mu}(\xi)=\int_{}^{}e^{i \xi(x)} \mathrm{d}\mu(x)$.}
of $M$, and $L_{\Phi^{\circ}}$ the Lebesgue measure on $\Phi^{\circ}$.

Now let us suppose that the continuous dual $V'$ is endowed with a symplectic form
$\sigma:V'\times V'\to \mathds{R}$, \emph{i.e.} an antisymmetric, non degenerate, and bilinear form. Then it
is possible to construct in a unique fashion the Weyl C*-algebra of canonical commutation relations
$\mathbf{W}(V',\sigma)$. Let us remark that we consider the Weyl C*-algebra to depend on a semiclassical
parameter $h$, appearing in the so-called Weyl's relations among the generators $\{W(\xi)\,,\, \xi\in V'\}$ of
the algebra:
\begin{equation}
  \label{eq:3}
  \forall \xi,\zeta\in V'\,,\, W(\xi)W(\zeta)=e^{-i h \sigma(\xi,\zeta)}W(\xi+\zeta)\; .
\end{equation}
From time to time, to avoid confusion, the dependence on $h$ of $\mathbf{W}_h(V',\sigma)$ will be made
explicit.

The regular states on $\mathbf{W}_h(V',\sigma)$ play an important role to study its semiclassical behavior
$h\to 0$; their definition is recalled below.
\begin{definition}[Regular states]
  \label{def:2}
  A positive element $\omega\in \mathbf{W}(V',\sigma)'$ of the continuous dual of the Weyl C*-algebra is a \emph{regular
    state} iff the $\mathds{R}$-action
  \begin{equation*}
    t\mapsto \omega\bigl(W(t\xi)\bigr)
  \end{equation*}
  is continuous for any $\xi\in V'$. A regular state is \emph{normalized} iff $\omega\bigl(W(0)\bigr)=1$.
\end{definition}

Given any $\omega\in \mathbf{W}(V',\sigma)'$, let us denote by $( H_{\omega}, \pi_{\omega}, \Omega_{\omega} )$ the corresponding GNS
representation of $\mathbf{W}(V',\sigma)$. Let us recall that $H_{\omega}$ is a Hilbert space,
\begin{equation*}
  \pi_{\omega}: \mathbf{W}(V',\sigma)\to \mathscr{B}(H_{\omega})
\end{equation*}
a *-homomorphism, and $\Omega_{\omega}\in H_{\omega}$ the canonical cyclic vector of the representation. If
$\omega$ is regular, it is possible to define -- by Stone's theorem -- the self-adjoint generator of the unitary
group $\pi_{\omega}\bigl(W(\mathds{R} \xi)\bigl)$ for any $\xi\in V'$. Let us denote such generator by
\begin{equation}
  \label{eq:4}
  \varphi_{\omega}(\xi)=\frac{\mathrm{d}\pi_{\omega}\bigl(W(t\xi)\bigl)}{\mathrm{d}t}\Bigr\rvert_{t=0}\; .
\end{equation}
For any finite dimensional symplectic subspace\footnote{If $V$ is locally convex, $F\subset V'$ is a subspace of
  finite dimension iff there exists $\Phi\in F(V)$ such that $F=\Phi^{\circ}$. Therefore we adopt directly the
  notation $\Phi^{\circ}$ for finite dimensional subspaces of $V'$.} $\Phi^{\circ}\subset V'$, the operators
$\varphi_{\omega}(\Phi^{\circ})$ share a common dense set of analytic vectors \citep[][Lemma
5.2.12]{MR1441540
}. In order to introduce the well-known creation and annihilation operators $a^{*}_{\omega}(\Phi^{\circ})$ and
$a_{\omega}(\Phi^{\circ})$, it is useful to see the symplectic $\Phi^{\circ}$ as a complex finite dimensional
inner product space. Let us denote by $J_{\sigma}: \Phi^{\circ}\to \Phi^{\circ}$, $J^2_{\sigma}=-1$, a linear
complex structure on $\Phi^{\circ}$ such that
\begin{equation}
  \label{eq:12}
  \sigma(J_{\sigma}\,\cdot\, ,J_{\sigma}\,\cdot \,)=\sigma(\,\cdot \,,\,\cdot \,)\; .
\end{equation}
Then $\Phi^{\circ}$ (of real dimension $2n$, $n\in \mathds{N}$) can be seen as a complex vector space (of complex
dimension $n$) defining
\begin{equation*}
  (a+ib)\xi=a\xi+bJ_{\sigma}\xi\; ;
\end{equation*}
and it is an inner product space by means of the Hermitian form
\begin{equation*}
  \sigma(\,\cdot \,,J_{\sigma}\,\cdot \,)+i\sigma(\,\cdot \,,\,\cdot \,)\; .
\end{equation*}
The space $\Phi^{\circ}$ is the polar of some $\Phi\in F(V)$, and therefore
\begin{equation*}
  V/\Phi \overset{\imath_{\sigma}}{\cong} (V/\Phi)'\overset{^{\mathrm{t}}p_{\Phi}}{\cong} \Phi^{\circ} \; .
\end{equation*}
The isomorphism $\imath_{\sigma}$ is chosen in a way such that for any $\xi\in \Phi^{\circ}$, and $w\in V/\Phi$,
\begin{equation*}
  \sigma\bigl(\xi,J_{\sigma}\, ^{\mathrm{t}}p_{\Phi} \,\circ\, \imath_{\sigma}(w)\bigr)= \, ^{\mathrm{t}}p_{\Phi}^{-1}(\xi)(w)\; .
\end{equation*}
The creation and annihilation operators are then defined for any $\xi\in \Phi^{\circ}$ by
\begin{equation}
  \label{eq:5}
  \begin{aligned}
    a^{*}_{\omega}(\xi)&=\frac{1}{2}\bigl(\varphi_{\omega}(\xi)-i\varphi_{\omega}(i\xi) \bigr)\; ,\\
    a_{\omega}(\xi)&=\frac{1}{2}\bigl(\varphi_{\omega}(\xi)+i\varphi_{\omega}(i\xi) \bigr)\; .
  \end{aligned}
\end{equation}
They are closed, and adjoint of each other when defined on
$D\bigl(\varphi_{\omega}(\xi)\bigr)\cap D\bigl(\varphi_{\omega}(i\xi)\bigl)$. The self-adjoint number operator
is a combination of the creation and annihilation operators along a (complex) orthonormal basis of
$\Phi^{\circ}$. Let $\{\zeta_j\}_{j=1}^n$ be an o.n.b. of $\Phi^{\circ}$, then the number operator associated
to $\Phi\in F(V)$ is defined by
\begin{equation}
  \label{eq:6}
  N_{\omega,\Phi}=\sum_{j=1}^na_{\omega}^{*}(\zeta_j)a_{\omega}(\zeta_j)\; .
\end{equation}
Finally, let us define the scale of Hilbert spaces $(H_{\omega,\Phi}^{\delta})_{\delta\in \mathds{R}}$ associated to the number
operator. For any $\delta\in \mathds{R}^+$, $H_{\omega,\Phi}^{\delta}=D(N^{\delta}_{\omega,\Phi})$ endowed with the
norm
\begin{equation*}
  \lVert \langle N_{\omega,\Phi} \rangle_{}^{\delta} \,\cdot\, \rVert_{H_{\omega}}^{}\; ;
\end{equation*}
$H_{\omega,\Phi}^{-\delta}$ is the completion of $H_{\omega}$ with respect to the norm
\begin{equation*}
  \lVert \langle N_{\omega,\Phi} \rangle_{}^{-\delta} \,\cdot\, \rVert_{H_{\omega}}^{}\; .
\end{equation*}
Here $\langle \,\cdot\, \rangle_{}$ stands for $(\lvert \,\cdot\, \rvert_{}^2+1)^{\frac{1}{2}}$. The scale of spaces
$(H_{\omega,\Phi}^\delta)_{\delta\in \mathds{R}}$ satisfies
\begin{equation*}
  \forall \delta\leq \delta '\;,\; H_{\omega,\Phi}^{\delta}\subseteq H_{\omega,\Phi}^{\delta '}\; .
\end{equation*}
Let us denote by $\mathscr{S}_{\omega,\Phi}$ the nuclear space
\begin{equation}
  \label{eq:7}
  \mathscr{S}_{\omega,\Phi}=\bigcap_{\delta\in \mathds{N}}H_{\omega,\Phi}^{\delta}\; ,
\end{equation}
and by $\mathscr{S}'_{\omega,\Phi}$ its continuous dual
\begin{equation}
  \label{eq:8}
  \mathscr{S}'_{\omega,\Phi}=\bigcup_{\delta\in \mathds{N}}H_{\omega,\Phi}^{-\delta}\; .
\end{equation}

The definitions above are natural to study the pseudodifferential calculus for cylindrical symbols on $V$
(with base $\Phi\in F(V)$ of codimension $2n$), for it reduces to the standard finite-dimensional
pseudodifferential calculus. In fact, the Weyl operators $\pi_{\omega}\bigl(W(\Phi^{\circ})\bigr)$, and the
creation and annihilation operators $a^{\#}_{\omega}(\Phi^{\circ})$ are equivalent to their counterparts on
the Schrödinger or Bargmann-Fock representation of the finite dimensional Weyl C*-algebra
$\mathbf{W}(\Phi^{\circ},\sigma)$; while $\mathscr{S}_{\omega,\Phi}$ plays the role of the Schwartz space
$\mathscr{S}(\mathds{R}^n)$, and $\mathscr{S}'_{\omega,\Phi}$ of the tempered distributions
$\mathscr{S}'(\mathds{R}^n)$. We assume that the reader is familiar with the techniques and results of
pseudodifferential calculus and semiclassical analysis in finite dimensions \citep[][are just a few of the
monographs on the subject]{MR983366
  ,HeNi
  , MR2304165
  ,MR2952218
}. As a reference, in the remaining of this section we formulate, in the cylindrical setting, the results that
we will use the most.

\begin{definition}[Weyl quantization of a cylindrical symbol]
  \label{def:3}
  Let $f$ be a cylindrical symbol on $V$ with base $\Phi\in F(V)$, such that $\hat{f}_{\Phi}\in L^1(\Phi^{\circ})$. Then the Weyl
  quantization $\mathrm{Op}_{\frac{1}{2}}^h(f)\in \mathbf{W}(V',\sigma)$ of $f$ is defined by the Bochner integral
  \begin{equation*}
    \mathrm{Op}_{\frac{1}{2}}^h(f)=\int_{\Phi^{\circ}}^{}\hat{f}_{\Phi} (\xi)W_h(2\pi\xi)\,  \mathrm{d}L_{\Phi^{\circ}}(\xi)\; .
  \end{equation*}

  On any regular representation $(H_{\omega},\pi_{\omega},\Omega_{\omega})$ of $\mathbf{W}(V',\sigma)$, and for any
  $f_{\Phi}\in \mathscr{S}'(V/\Phi)$, the Weyl quantization
  $\pi_{\omega}\bigl(\mathrm{Op}_{\frac{1}{2}}^h(f)\bigr)$ is a continuous map from
  $\mathscr{S}_{\omega,\Phi}$ to $\mathscr{S}'_{\omega,\Phi}$ defined by
  \begin{equation*}
    \forall \psi,\varphi\in \mathscr{S}_{\omega,\Phi}\; ,\;  \bigl\langle\psi,   \pi_{\omega}\bigl(\mathrm{Op}_{\frac{1}{2}}^h(f)\bigr)\varphi\bigr\rangle_{H_{\omega}}= \bigl(\, \hat{f}_{\Phi}\, , \, \bigl\langle \psi,\pi_{\omega}(W_h(2\pi\,\cdot\, ))\varphi\bigr\rangle_{H_{\omega}} \,\bigr)_{\mathscr{S}'\times \mathscr{S}}\; .
  \end{equation*}
\end{definition}

In Definition~\ref{def:3}, we stressed the dependence on the semiclassical parameter $h$ of the Weyl
quantizations (through the Weyl operators $W_h(2\pi\xi)$). The semiclassical behavior of regular quantum
states is obtained studying the corresponding generating functional. The generating functional of a state
$\omega$ is defined for any $\xi\in V'$ as
\begin{equation*}
  G_{\omega}(\xi)=\omega\bigl(W(\xi)\bigr)\; .
\end{equation*}
The generating functional defines uniquely regular states, \emph{i.e.} there is a bijection between regular
states and functions $G:V'\to \mathds{C}$ that are continuous when restricted to any finite dimensional
subspace of $V$, and that are ``almost'' of positive type:
\begin{equation*}
  \sum_{i,j\in F}^{}G(\xi_i-\xi_j)e^{ih\sigma(\xi_i,\xi_j)}\bar{\lambda}_j\lambda_i\geq 0\;,
\end{equation*}
for any finite index set $F$, complex numbers $\{\lambda_i\}_{i\in F}$, and $\{\xi_i\}_{i\in F}\subset V'$
\citep{MR0128839
}. On the other hand, there is a bijection (the Fourier transform) between cylindrical measures on $V$ and
functions $\Gamma:V'\to \mathds{C}$ that are continuous when restricted to any finite dimensional subspace of
$V$, and of positive type:
\begin{equation*}
  \sum_{i,j\in F}^{}\Gamma(\xi_i-\xi_j)\bar{\lambda}_j\lambda_i\geq 0\;.
\end{equation*}
Therefore the topology of simple convergence on $\mathds{C}^{V'}$ induces a topology $\mathfrak{T}$ on the
disjoint union of the sets of regular quantum states on $\mathbf{W}_h(V',\sigma)$, $h>0$, together with the
cylindrical measures on $V$ \citep[see][]{Falconi:2016ab
}. The Wigner or semiclassical measures associated to a generalized sequence
$(\omega_{h_{\beta}})_{\beta\in B}$ of regular states with no loss of mass are its $\mathfrak{T}$-cluster
points as $h_{\beta}\to 0$ (more precisely, they are the $\mathfrak{P}\vee \mathfrak{T}$-cluster points, where
$\mathfrak{P}$ is the topology induced by the ultraweak convergence of the finite dimensional measures, and
$\mathfrak{P}\vee \mathfrak{T}$ is the upper bound topology). The ``no loss of mass condition'' is automatically
satisfied if we assume either Hypothesis~\ref{hyp:2} or~\ref{hyp:4}.
\begin{proposition}[\citep{Falconi:2016ab
  }]
  \label{prop:1}
  Let $(\omega_{h_{\gamma}})_{\gamma\in C}$, $h_{\gamma}\to 0$, be a uniformly bounded net of regular states on
  $\bigl(\mathbf{W}_{h_{\gamma}}(V',\sigma)\bigr)_{\beta\in B}$ with no loss of mass.
  Then the set $\mathscr{W}(\omega_{h_{\gamma}},\gamma\in C)$ of its $\mathfrak{T}$-cluster points is not empty, and each cluster
  point is a cylindrical measure on $V$, called a \emph{Wigner measure}. Conversely, for any cylindrical
  measure $M$ on $V$ there is at least one net $(\omega^{(M)}_{h_{\gamma}})_{\gamma\in C}$ of regular quantum
  states on $\bigl(\mathbf{W}_{h_{\gamma}}(V',\sigma)\bigr)_{\gamma\in C}$ that $\mathfrak{T}$-converges to it.
\end{proposition}

For any $\Phi\in F(V)$, we can define the topologies $\mathfrak{P}_{\Phi}$ and $\mathfrak{T}_{\Phi}$ on the restrictions
to $\mathbf{W}(\Phi^{\circ},\sigma)$ of regular quantum states on $\mathbf{W}(V',\sigma)$, together with the
Radon measures on $V/\Phi$. We then have the following characterization of semiclassical measures.

\begin{lemma}
  \label{lemma:1}
  \begin{equation*}
    \omega_{h_n}\bigr\rvert_{\mathbf{W}_{h_n}(\Phi^{\circ},\sigma)}\overset{h_{n}\to 0}{\longrightarrow}_{\mathfrak{P}_{\Phi}}\mu_{\Phi}
  \end{equation*}
  iff for any cylindrical symbol on $V$ with $f_{\Phi}\in C_0^{\infty}(V/\Phi)$
  \begin{equation*}
   \lim_{n\to \infty} \omega_{h_n}\bigl( \mathrm{Op}_{\frac{1}{2}}^{h_n}(f)\bigr)=\int_{V/\Phi}^{}f_{\Phi}(w)  \mathrm{d}\mu_{\Phi}(w)\; .
 \end{equation*}
 If in addition there is no loss of mass, then
 \begin{equation*}
   \omega_{h_n}\bigr\rvert_{\mathbf{W}_{h_n}(\Phi^{\circ},\sigma)}\overset{h_{n}\to 0}{\longrightarrow}_{\mathfrak{P}_{\Phi}\vee \mathfrak{T}_{\Phi}}\mu_{\Phi}\; .
 \end{equation*}
\end{lemma}

On any regular representation $(H_{\omega},\pi_{\omega},\Omega_{\omega})$, it is possible to extend the only if part of
Lemma~\ref{lemma:1} to polynomially bounded symbols. Let us denote by $\mathfrak{S}^p_+(H)$,
$1\leq p\leq \infty$, the cone of positive continuous operators on the Hilbert space $H$ that belong to the
Schatten ideal of order $p$; and by $S_{\Phi}(\langle w \rangle_{\Phi}^s,g_{\Phi})$, $s\in \mathds{R}$, the
Hörmander class of symbols on $V/\Phi$, with the metric $g_{\Phi}$ to be either
$\lvert \mathrm{d}w \rvert_{\Phi}^2$ or $\frac{\lvert \mathrm{d}w \rvert_{\Phi}^2}{\langle w \rangle_{\Phi}^2}$.
\begin{proposition}
  \label{prop:2}
  Let $\Phi^{\circ}\subset V'$ be a symplectic subspace. In addition, let $(\omega_{h_{\gamma}})_{\gamma\in C}$ be a net of regular states on
  $\bigl(\mathbf{W}_{h_{\gamma}}(V',\sigma)\bigr)_{\gamma\in C}$, and
  $\varrho_{h_{\gamma}}\in \mathfrak{S}^1_+(H_{\omega})$ for any $\gamma\in C$, such that for some
  $\delta\in \mathds{R^+}$
  \begin{equation*}
    \Tr\bigl(\langle N_{\omega,\Phi}  \rangle_{}^{\delta}\varrho_{h_{\gamma}} \bigr)_{H_{\omega}}\leq C_{\delta}<\infty\; .
  \end{equation*}
  Then
  \begin{equation*}
    \varrho_{h_{\gamma}}\to_{\mathfrak{T}} \mu=(\mu_{\Psi})_{\Psi\in F(V)}
  \end{equation*}
  implies that for any cylindrical $f$ on $V$ with base $\Phi$ and $f_{\Phi}\in S_{\Phi}(\langle w \rangle_{}^\delta,g_{\Phi})$,
  \begin{equation*}
    \lim_{\gamma\in C}\Tr\bigl(\pi_{\omega}\bigl(\mathrm{Op}_{\frac{1}{2}}^{h_{\gamma}}(f)\bigr) \varrho_{h_{\gamma}}  \bigr)=\int_{V/\Phi}^{}f_{\Phi}(w)  \mathrm{d}\mu_{\Phi}(w)\; .
  \end{equation*}
\end{proposition}

\section{Concentration of Wigner measures}
\label{sec:conc-wign-meas}

As explained in Section~\ref{sec:cylindr-quant-regul}, to any bounded sequence
$(\omega_{h_{\gamma}})_{\gamma\in C}$, $h_{\gamma}\to 0$, of regular states on
$\bigl(\mathbf{W}_{h_{\gamma}}(V',\sigma)\bigr)_{\gamma\in C}$ there corresponds a nonempty set of cylindrical
measures on $V$, called Wigner measures. In applications, it is desirable to have sufficient conditions on
$(\omega_{h_{\gamma}})_{\gamma\in C}$ such that each corresponding Wigner measure is concentrated as a true
Radon measure on some given topological vector space $W$ (that may or may not differ from $V$). The main
purpose of this note is to provide such sufficient conditions, exploiting cylindrical semiclassical
analysis. Any cylindrical measure $M=(\mu_{\Phi})_{\Phi\in F(V)}$ on a locally convex space $V$ is a Borel Radon measure on
the space
\begin{equation*}
  \overline{V}=  \prod_{\Phi\in F(V)}\overline{V/\Phi}\; ,
\end{equation*}
product of \v{C}ech compactifications of the finite dimensional quotients, endowed with the product
topology. The cylindrical measure $M\in \mathcal{M}_{\mathrm{rad}}(\overline{V})$ concentrates as the Borel
Radon measure $\mu\in \mathcal{M}_{\mathrm{rad}}(X)$ on the real topological vector space $X$ iff there exists an
injection $\varphi:X\to \overline{V}$,
\begin{equation*}
  M\bigl(\overline{V}\setminus \varphi X\bigr)=0\; ,
\end{equation*}
and $M(b)=\mu(\varphi^{-1}b)$ for any Borel set $b\in \mathscr{B}(\overline{V})\cap \varphi X$. If $X$ is a complex topological
vector space, the injection is $\varphi:\mathbb{F}X\to \overline{V}$, where $\mathbb{F}$ is the forgetful
functor from the category of complex vector spaces to the category of real vector spaces. The rest of the
definition is modified accordingly.

\subsection{Separable Hilbert spaces}
\label{sec:separ-hilb-spac}

Let $X$ be a \emph{complex} vector space, and
\begin{equation*}
  t: X\times X\to \mathds{C}
\end{equation*}
a sesquilinear form on $X$. Let
\begin{equation*}
  q(x)=t(x,x)
\end{equation*}
be the corresponding quadratic form. Then $t$ factors through $X/\ker q$, \emph{i.e.} there is a sesquilinear
form $t_q:(X/\ker q) \times (X/\ker q) \to \mathds{C}$ such that
\begin{equation*}
  t_q([x]_q,[y]_q)=t(x,y)\; .
\end{equation*}
With an abuse of notation, let us still denote $t_q$ by $t$. The completion $X_q$ of $X/\ker q$ with respect
to the norm $\sqrt{q} $ is a complex Hilbert space with inner product $t$. Throughout this section, we suppose
that $X_q$ is \emph{separable}. The results of this section can be easily adapted to real vector spaces.

Let us denote by $\mathbb{F}$ the forgetful functor from the category of complex vector spaces to the category of real
vector spaces. Therefore $\mathbb{F}X_q$ is a real Hilbert space with inner product $\Re \, t$, and a symplectic
space with symplectic form $\Im \, t$.
In addition, if $\{x_1,\dotsc, x_n,\dotsc\}$ is a complex orthonormal basis of the space $X_q$, then
$\{\mathbb{F} x_1,\mathbb{F} (i x_1),\dotsc,\mathbb{F} x_n,\mathbb{F}(ix_n),\dotsc\}$ is a real orthonormal
basis of the space $\mathbb{F} X_q$. We make the following assumption of compatibility between the inner
product of $X$ and the symplectic form $\sigma$ of $V'$, dual to a locally convex space $V$ as in
Section~\ref{sec:cylindr-quant-regul}.
\begin{hyp}
  \label{hyp:1}
  $\mathbb{F}X\subseteq V'$. In addition:
  \begin{itemize}
  \item The complex structure $J:\mathbb{F}X\to \mathbb{F}X$, defined by $J\mathbb{F}x=\mathbb{F}(ix)$, is such that
    \begin{equation*}
      \sigma(J\,\cdot\,,J\,\cdot \, )=\sigma(\,\cdot \,,\,\cdot \,)\; .
    \end{equation*}
  \item There exists a linear map $\mathbf{m}:X\to X$ such that for any $x,y\in X$,
    \begin{equation*}
      t(x,y)=\sigma\bigl(\mathbb{F}\mathbf{m}(x), \mathbb{F}\mathbf{m}(iy)\bigr)+i\sigma\bigl(\mathbb{F}\mathbf{m}(x),\mathbb{F}\mathbf{m}(y)\bigr)\; .
  \end{equation*}
  \end{itemize}
\end{hyp}


Now let
\begin{equation*}
  \{[e_n]_q,n\in \mathds{N}\}\subset X/\ker q
\end{equation*}
be an orthonormal basis of $X_q$. In addition, let
\begin{equation*}
  \bigl\{\varepsilon_n=\mathbb{F}\mathbf{m}(e_n); n\in \mathds{N}\bigr\}\subset V'\; .
\end{equation*}
Let us consider any finite-dimensional subspace of $X$ of the form
\begin{equation*}
  E_R=\mathrm{span}_{\mathds{C}}\{e_0,e_1,\dotsc,e_R\}\; ,\; R\in \mathds{N}\, .
\end{equation*}
The complex Hilbert space $E_R$ is isomorphic to $p_R X_q$, where
\begin{equation*}
  p_R=\sum_{j=1}^Rp([e_j]_q)
\end{equation*}
is the sum of rank one orthogonal projections. We define the symplectic subspace $\Phi^{\circ}_{E_R}\subset V'$ as
\begin{equation*}
  \Phi^{\circ}_{E_R} =\mathrm{span}_{\mathds{R}}\{\varepsilon_0,J\varepsilon_0,\dotsc,\varepsilon_R,J\varepsilon_R\}\; .
\end{equation*}
As in Section~\ref{sec:cylindr-quant-regul}, we remark that $\Phi^{\circ}_{E_R}$ can be seen as a complex inner
product space with inner product
\begin{equation*}
  b(\,\cdot \,,\,\cdot \,)=\sigma(\,\cdot \,,J\,\cdot \,)+i\sigma(\,\cdot \,,\,\cdot \,)\; .
\end{equation*}
By construction,
\begin{equation*}
  (\Phi^{\circ}_{E_R},b) \overset{i_t}{\cong}(E_R,t)\; ,
\end{equation*}
as complex Hilbert spaces. Now let $(H_{\omega},\pi_{\omega},\Omega_{\omega})$ be the GNS representation of
$\mathbf{W}(V',\sigma)$ corresponding to a regular state $\omega$. We define the $\mathbf{m}$-creation and
annihilation operators as follows. For any $x\in X$,
\begin{equation}
  \label{eq:18}
  \begin{aligned}
    \mathbf{a}_{\omega}^{*}(x)&=\frac{1}{2}\Bigl(\varphi_{\omega}\bigl(\mathbb{F}\mathbf{m}(x)\bigr)-i\varphi_{\omega}\bigl(\mathbb{F} \mathbf{m}(ix)\bigr)\Bigr)\; ,\\
    \mathbf{a}_{\omega}(x)&=\frac{1}{2}\Bigl(\varphi_{\omega}\bigl(\mathbb{F}\mathbf{m}(x)\bigr)+i\varphi_{\omega}\bigl(\mathbb{F} \mathbf{m}(ix)\bigr)\Bigr)\; .
  \end{aligned}
\end{equation}
The corresponding $\mathbf{m}$-number operator on $\Phi^{\circ}_{E_R}$ is the generalization of the number operator to
$\mathbf{m}$-creation and annihilation:
\begin{equation}
  \label{eq:17}
  \mathbf{N}_{\omega,R}=\sum_{j=1}^R\mathbf{a}^{*}_{\omega}(\mathbf{m}(e_j))\mathbf{a}_{\omega}(\mathbf{m}(e_j))\; .
\end{equation}
For any $\delta>0$, the sequence $(\mathbf{N}_{\omega,R}^{\delta})_{R\in \mathds{N}}$ is a monotonically increasing sequence of
positive self adjoint operators. If we denote by $(\mathbf{n}_{\omega,R}^{\delta})_{R\in \mathds{N}}$ the
sequence of corresponding quadratic forms, then
\begin{equation*}
\mathbf{n}_{\omega,\infty}^{\delta} =\sup_{R\in \mathds{N}} \mathbf{n}_{\omega,R}^{\delta}
\end{equation*}
is a closed positive quadratic form on $H_{\omega}$ (not necessarily densely defined), with domain of definition
\begin{equation*}
  D\bigl(\mathbf{n}_{\omega,\infty}^{\delta}\bigr)=\Bigl\{\psi\in \bigcap_{R\in \mathds{N}}D\bigl(\mathbf{n}_{\omega,R}^{\delta}\bigr), \sup_{R\in \mathds{N}}\mathbf{n}_{\omega,R}^{\delta}(\psi)<\infty\Bigr\}\; .
\end{equation*}
Now let $(\psi_{h_{\gamma}})_{\gamma\in C}$, $h_{\gamma}\to 0$, be a net of vectors in $H_{\omega_{h_{\gamma}}}$ for any
$\gamma\in C$, and denote by $\mathscr{W}(\psi_{h_{\gamma}},\gamma\in C)$ the set of Wigner cylindrical measures
on $V$ associated to it -- any vector in $H_{\omega}$ has an associated rank-one orthogonal projection that is
a regular state on $\mathbf{W}(V',\sigma)$. The aim is to prove that the following condition is sufficient to
prove that any $M\in \mathscr{W}(\psi_{h_{\gamma}},\gamma\in C)$ concentrates as a (finite) Borel Radon measure
$\mu$ on the Hilbert space $X_q$.
\begin{hyp}
  \label{hyp:2}
  Let $(\psi_{h_{\gamma}})_{\gamma\in C}$ be a net of vectors on regular GNS representations of
  $\mathbf{W}_{h_{\gamma}}(V',\sigma)$, with
  \begin{equation*}
    \sup_{\gamma\in C}\lVert \psi_{h_{\gamma}}  \rVert_{H_{\omega_{h_{\gamma}}}}^{}<\infty\; .
  \end{equation*}
  There exists $\delta>0$ such that $\psi_{h_{\gamma}}\in D\bigl(\mathbf{n}_{\omega_{h_{\gamma}},\infty}^{\delta}\bigr)$ for any
  $\gamma\in C$, and there exists $\mathcal{K}>0$ such that uniformly in $\gamma\in C$,
  \begin{equation*}
    \mathbf{n}_{\omega,\infty}^{\delta}(\psi_{h_{\gamma}})\leq \mathcal{K}\; .
  \end{equation*}
\end{hyp}

First of all, any $M\in \mathscr{W}(\psi_{h_{\gamma}},\gamma\in C)$ is a cylindrical measure on $V$. Therefore
$M=(\mu_{\Phi})_{\Phi\in F(V)}$. Now let us consider
$(\mu_{\Phi_R})_{R\in \mathds{N}}\subset (\mu_{\Phi})_{\Phi\in F(V)}$, where
$\Phi_R=\Phi_{E_R}^{ \circ \circ}$ is the polar of $\Phi_{E_R}^{\circ}$. $\Phi_R$ belongs to $F(V)$, and satisfies
\begin{equation*}
  \Phi_R^{\circ}=\Phi_{E_R}^{\circ \circ \circ}= \Phi_{E_R}^{\circ}\; ,
\end{equation*}
via the bipolar theorem \citep{MR633754
}. Therefore by means of the isomorphisms (see Section~\ref{sec:cylindr-quant-regul})
\begin{equation*}
  V/\Phi_R\overset{\imath_{\sigma}}{\cong} (V/\Phi_R)'\overset{^{\mathrm{t}}p_{\Phi_R}}{\cong}\Phi_R^{\circ}\overset{i_t}{\cong} E_R\; ,
\end{equation*}
it is possible to identify each $V/\Phi_R$ with $E_R$, \emph{i.e.} with a complex finite dimensional Hilbert
subspace of $X_q$. Using the aforementioned identification, it follows that
\begin{equation*}
  \forall R\in \mathds{N},\; V/\Phi_R\subset V/\Phi_{R+1}\; ;\quad \bigcup_{R\in \mathds{N}}V/\Phi_R=X_q\; .
\end{equation*}
Finally, for any $j\in \mathds{N}$ the orthogonal projections
\begin{equation*}
  p_{R+j,R}:E_{R+j}\to E_R
\end{equation*}
coincide with the projections
\begin{equation*}
  p_{\Phi_{R+j},\Phi_R}:V/\Phi_{R+j}\to V/\Phi_R
\end{equation*}
relative to the projective family $M=(\mu_{\Phi})_{\Phi\in F(V)}$. Therefore
$(\mu_{\Phi_R})_{R\in \mathds{N}}$ is a projective subfamily of $M$. As proved in
\citep{MR0466482
}, $M$ is concentrated as a Borel Radon measure $\mu$ on $X_q$ iff
\begin{equation}
  \label{eq:9}
\forall R\in \mathds{N},\quad  \lim_{r\to \infty}\mu_{\Phi_R}\bigl(\{f\in V/\Phi_R\,,\, q(f)\geq r\}\bigr)=0\; ,
\end{equation}
where $q$ is the Euclidean norm on $V/\Phi_R$ (seen as a complex space), that coincides with the Hilbert norm of
$X_q$. In fact, if we denote by
\begin{equation*}
  \mu_{E_R}= i_t \,_{*}\, ^{\mathrm{t}}p_{\Phi_R}\, _{*}\, \imath_{\sigma}\, _{*}\, \mu_{\Phi_R}
\end{equation*}
the pushforward image of $\mu_{\Phi_R}$, \eqref{eq:9} is equivalent to
\begin{equation}
  \label{eq:22}
  \lim_{r\to \infty}\mu_{E_R}\bigl(\{f\in E_M\,,\, q(f)\geq r\}\bigr)=0\; .
\end{equation}
Suppose now that we have proved that $M$ concentrates as a Radon measure $\mu$ on $X_q$. In addition, let the
following condition be satisfied:
\begin{equation}
  \label{eq:10}
\exists \delta>0\,,\,\exists \mathcal{K}>0\,,\quad  \int_{X_q}^{}q(f)^{\delta}  \mathrm{d}\mu(f)\leq \mathcal{K}\; .
\end{equation}
Equation \eqref{eq:10} implies that $\hat{\mu}:X_q\to \mathds{C}$ is continuous. Let $0<\delta\leq 1$, and let
$g_1,g_2\in X_q$. Then for some $\mathcal{K}_{\delta}>0$
\begin{equation*}
  \begin{split}
    \bigl\lvert \hat{\mu}(g_1)-\hat{\mu}(g_2)  \bigr\rvert_{}^{}\leq \mspace{-5mu}\int_{X_q}^{}\mspace{-7mu}\bigl\lvert e^{i\, \Re \, t(g_1-g_2,f)}-1  \bigl\rvert_{}^{}  \mathrm{d}\mu(f)\leq \mathcal{K}_{\delta}\mspace{-5mu}\int_{X_q}^{}\mspace{-7mu}\bigl\lvert t(g_1-g_2,f)  \bigr\rvert_{}^{\delta}  \mathrm{d}\mu(f)\\\leq \mathcal{K}_{\delta}q(g_1-g_2)^{\delta}\int_{X_q}^{}q(f)^{\delta}  \mathrm{d}\mu(f)\leq \mathcal{K}_{\delta}\mathcal{K}\, q(g_1-g_2)^{\delta}\; .
  \end{split}
\end{equation*}
If $\delta>1$, then
\begin{equation*}
  \begin{split}
    \int_{X_q}^{}\mspace{-7mu}q(f)  \mathrm{d}\mu(f)= \mspace{-5mu}\int_{X_q}^{}\mspace{-11mu}1_{\{q(\cdot )\leq 1\}}(f)q(f)  \mathrm{d}\mu(f)+\mspace{-5mu}\int_{X_q}^{}\mspace{-11mu}1_{\{q(\cdot )>1\}}(f)q(f)  \mathrm{d}\mu(f)\\
    \leq \mu(X_q)+\mathcal{K}\; ,
  \end{split}
\end{equation*}
and therefore the result above for $\delta=1$ applies with constant $\mu(X_q)+\mathcal{K}$.
\begin{thm}
  \label{thm:1}
  Let $V$ be a locally convex space with dual $(V',\sigma)$ endowed with a symplectic form. In addition, let
  $X$ be a complex vector space with a sesquilinear form $t$ such that the compatibility condition \ref{hyp:1}
  with $\sigma$ holds.

  For a cylindrical measure $M\in \mathscr{W}(\psi_{h_{\gamma}},\gamma\in C)$ to be concentrated as a Borel Radon measure
  $\mu$ on the separable Hilbert space $X_q$, it is sufficient that Hypothesis \ref{hyp:2} is satisfied. The
  Fourier transform $\hat{\mu}$ is furthermore continuous on $X_q$; and if $\tilde{\delta} >0$ and
  $\tilde{\mathcal{K}}>0$ are the values for which \ref{hyp:2} holds, then
  \begin{equation*}
    \int_{X_q}^{}q(f)^{\tilde{\delta}}  \mathrm{d}\mu(f)\leq \tilde{\mathcal{K}}\; .
  \end{equation*}
\end{thm}
\begin{proof}
  Hypothesis~\ref{hyp:2} is sufficient to have no loss of mass (and hence convergence of the generating
  functional) when the state is restricted to each $\Phi_R^{\circ}$. We have to prove that
  Hypothesis~\ref{hyp:2} also implies both \eqref{eq:9} and \eqref{eq:10}. First of all, let us prove
  \eqref{eq:9}. As discussed in Section~\ref{sec:motivation}, the $\mathbf{m}$-creation and annihilation
  operators establish a correspondence with the standard pseudodifferential calculus associated to the
  Heisenberg group $\mathbf{H}(\Phi^{\circ}_{E_R},\sigma)$, and that is assured by the compatibility condition
  \ref{hyp:1}. Then \eqref{eq:9} is proved using the argument introduced in \citep[][Theorem
  6.2]{ammari:nier:2008
  }. Let $(\psi_{h_{\gamma}})_{\gamma\in C}$ be a net of vectors in regular GNS representations, and let
  $\Phi^{\circ}_{E_R}$ be defined as above for some $R\in \mathds{N}$. Then Hypothesis \ref{hyp:2} yields that
  \begin{equation}
    \label{eq:13}
    \langle \psi_{h_{\gamma}}  , \mathbf{N}_{\omega_{h_{\gamma}},R}^{\tilde{\delta}}\,\psi_{h_{\gamma}} \rangle_{H_{\omega_{h_{\gamma}}}}\leq \tilde{\mathcal{K}}
  \end{equation}
  uniformly in $\gamma$ for some $\tilde{\delta}>0$. By means of the link between Wick and Weyl quantization, it is
  well-known that
  \begin{equation*}
    1+\tfrac{h R}{2}+\mathbf{N}_{\omega_h,R}= \mathbf{Op}^h_{\frac{1}{2}}\bigl(1+q(\cdot )\bigr) \; ,
  \end{equation*}
  where $q: V/\Phi_R\to \mathds{R}^+$ is the euclidean squared norm of the finite dimensional (complex) Hilbert
  space $V/\Phi_R\cong E_R$, and $\mathbf{Op}^h_{\frac{1}{2}}$ is the Weyl quantization procedure defined in
  Section~\ref{sec:cylindr-quant-regul}, but with each $W_h(\mathbb{F}x)$, $x\in X$, replaced by
  $W_h(\mathbb{F}\mathbf{m}(x))$. The operator
  \begin{equation*}
    \mathbf{Op}^h_{\frac{1}{2}}\bigl(1+q(\cdot ) \bigr)
  \end{equation*}
  is fully elliptic \citep[see
  e.g.][]{MR743094
  }. Therefore for any $s\in \mathds{R}$,
  \begin{equation}
    \label{eq:14}
    \bigl(1+\tfrac{h R}{2}+\mathbf{N}_{\omega_h,R}\bigr)^s=\mathbf{Op}^h_{\frac{1}{2}}\bigl((1+q(\cdot ))^s+O(h) \bigr)\; ,
  \end{equation}
  where the function $h^{-1}O(h)$ is uniformly bounded in the Hörmander class
  \begin{equation*}
    S_{R}\Bigl(\langle w \rangle_{\Phi_R}^{s-2},\tfrac{\lvert \mathrm{d}w \rvert_{\Phi_R}^2}{\langle w \rangle_{\Phi_R}^2}\Bigr)\; .
  \end{equation*}
  For any $h_{\gamma}\leq R^{-1}$, \eqref{eq:13} implies there exists $\tilde{\mathcal{K}}_{\tilde{\delta}}>0$ such that
  \begin{equation*}
    \bigl\langle \psi_{h_{\gamma}}  , \bigl(1+\tfrac{h_{\gamma} R}{2}+\mathbf{N}_{\omega_{h_{\gamma}},R}\bigr)^{\tilde{\delta}}\psi_{h_{\gamma}} \bigr\rangle_{H_{\omega_{h_{\gamma}}}}\leq \tilde{\mathcal{K}}_{\tilde{\delta}}\; .
  \end{equation*}
  Now let $r\geq 1$, and $C_R\subseteq C$ such that for any $\gamma\in C_R$, $h_{\gamma}\leq R^{-1}$. Using standard pseudodifferential
  techniques in $V/\Phi_R$
  \citep[see][]{ammari:nier:2008
  }, it is possible to prove that for any $\gamma\in C_R$ and $\delta>0$, there exists a $K>0$ such that the following
  operator inequality holds on $H_{\omega_{h_{\gamma}}}$:
  \begin{equation*}
    \mathbf{Op}_{\frac{1}{2}}^{h_{\gamma}}\bigl(\chi(r^{-\frac{1}{2}}\,\cdot \,)\bigr)\leq (1+r)^{-\delta}  (1+2Kh_{\gamma})\; \mathbf{Op}^{h_{\gamma}}_{\frac{1}{2}}\bigl((1+q(\cdot ))^{-\delta/2} \bigr)^{-2}\; .
  \end{equation*}
  Here $\chi(\cdot )\in C^{\infty}(V/\Phi_R)$ is non-negative and such that $\chi=0$ in a neighbourhood of
  $\{f\in V/\Phi_R\,,\, \sqrt{q(f)}\leq 1\}$. In addition, \eqref{eq:14} yields there exists $K'>0$ such that
  \begin{equation*}
    \Bigl\lVert \mathbf{Op}^{h_{\gamma}}_{\frac{1}{2}}\bigl((1+q(\cdot ))^{-\delta/2} \bigr)^{-2}\bigl(1+\tfrac{h_{\gamma} R}{2}+\mathbf{N}_{\omega_{h_{\gamma}},R}\bigr)^{\delta}  \Bigr\rVert_{\mathscr{B}(H_{\omega_{h_{\gamma}}})}^{}\leq K'h_{\gamma}+1\; .
  \end{equation*}
  Therefore for any $\gamma\in C_M$ and $r\geq 1$, there exists $\tilde{K}_{\tilde{\delta}}>0$ such that
  \begin{equation*}
    \bigl\langle \psi_{h_{\gamma}}  , \mathbf{Op}_{\frac{1}{2}}^{h_{\gamma}}\bigl(\chi(r^{-\frac{1}{2}}\,\cdot \,)\bigr) \psi_{h_{\gamma}}\bigr\rangle_{H_{\omega_{h_{\gamma}}}}\leq \tilde{K}_{\tilde{\delta}}(1+r)^{-\tilde{\delta}}\; .
  \end{equation*}
  Now let $M\in \mathscr{W}(\psi_{h_{\gamma}},\gamma\in C)$, and
  $(\psi_{h_{\gamma}})_{\gamma\in \underline{C}}$ the subnet of $(\psi_{h_{\gamma}})_{\gamma\in C_M}$ whose
  restrictions to $\bigl(\mathbf{W}_{h_{\gamma}}(\Phi^{\circ}_R,\sigma)\bigr)_{\gamma\in C_M}$,
  $R\in \mathds{N}$, converge to $\mu_{\Phi_R}$ with respect to the $\mathfrak{T}_{\Phi_R}$ topology. Then by
  Proposition~\ref{prop:2},
  \begin{equation*}
    \lim_{\gamma\in \underline{C}}\,\bigl\langle \psi_{h_{\gamma}}  , \mathbf{Op}_{\frac{1}{2}}^{h_{\gamma}}\bigl(\chi(r^{-\frac{1}{2}}\,\cdot \,)\bigr) \psi_{h_{\gamma}}\bigr\rangle_{H_{\omega_{h_{\gamma}}}}\mspace{-10mu}=\int_{V/\Phi_R}^{}\chi(r^{-\frac{1}{2}}f )  \mathrm{d}\mu_{\Phi_R}(f)\; .
  \end{equation*}
  Hence it follows that
  \begin{equation*}
    \begin{split}
      \mu_{\Phi_R}\bigl(\{f\in V/\Phi_R\,,\, q(f)\geq r\}\bigr)\leq \int_{V/\Phi_R}^{}\chi(r^{-\frac{1}{2}}f )  \mathrm{d}\mu_{\Phi_R}(f)\\
      =\lim_{\gamma\in \underline{C}}\,\bigl\langle \psi_{h_{\gamma}}  , \mathbf{Op}_{\frac{1}{2}}^{h_{\gamma}}\bigl(\chi(r^{-\frac{1}{2}}\,\cdot \,)\bigr) \psi_{h_{\gamma}}\bigr\rangle_{H_{\omega_{h_{\gamma}}}}\leq \tilde{K}_{\tilde{\delta}}(1+r)^{-\tilde{\delta}}\; .
    \end{split}
  \end{equation*}
  It remains to prove \eqref{eq:10} for $\tilde{\delta}$ and $\tilde{\mathcal{K}}$. Let
  $\mathbf{1}_0\in C_0^{\infty}(V/\Phi_R)$ be a smooth indicator function of a neighbourhood of $0$. The
  following inequality, for any $\delta>0$ and $r\geq 1$, and some $\mathcal{K}_R>0$, is again a consequence
  of standard pseudodifferential techniques:
  \begin{equation*}
    (\mathbf{N}_{\omega_{h},R})^{\delta}\geq \mathbf{Op}_{\frac{1}{2}}^h\bigl(q(\cdot )^{\delta}\mathbf{1}_0(r^{-1}\,\cdot\, )\bigr)- \mathcal{K}_R h \, (\mathbf{N}_{\omega_{h},R})^{\delta}
  \end{equation*}
  Reasoning as above, and using the fact that by \eqref{eq:9} $M$ concentrates as a Radon measure $\mu$ on
  $X_q$, we obtain
  \begin{equation*}
    \int_{E_R}^{}q(f)^{\tilde{\delta}}  \mathrm{d}\mu(f)\leq \tilde{\mathcal{K}}\; .
  \end{equation*}
  Now since the bound is uniform with respect to $R\in \mathds{N}$, we can take the limit $R\to \infty$ using dominated
  convergence.
\end{proof}

\subsection{Duals of nuclear spaces.}
\label{sec:dual-nuclear-space}

Now, let us suppose that $V'\supset \mathscr{N}$, where $\mathscr{N}$ is a nuclear barrelled space. Then
concentration of Wigner measures on $\mathscr{N}'$, endowed with any topology $\mathfrak{R}$ between
$\sigma(\mathscr{N}',\mathscr{N})$ and the Mackey topology $\tau(\mathscr{N}',\mathscr{N})$, admits an easy
characterization.

It is well-known that a cylindrical measure $M$ on $V$ is concentrated as a Radon measure $\mu$ on
$\mathscr{N}'_{\mathfrak{R}}$ iff $(\mu_{\Phi})_{\Phi\in F(\mathscr{N}'_{\mathfrak{R}})}$ can be identified
with a projective subfamily of $M$ and the restriction of $\hat{M}$ to $\mathscr{N}$ is continuous \citep[see
e.g.][]{MR0276436
}. Continuity properties of the Fourier transform of $M$ are inherited by continuity properties of the family
of generating functionals of regular states $\omega_{h_{\gamma}}$, $\gamma\in C$. In particular, equicontinuity of
\begin{equation}
  \label{eq:15}
  \mathbf{G}_B=\bigl\{G_{\omega_{h_{\gamma}}}\bigr\rvert_{\mathscr{N}}\,,\,\gamma\in C\bigr\}
\end{equation}
is sufficient to guarantee the continuity of its cluster points (in the topology of simple convergence), as
$h_{\gamma}\to 0$. This is in turn sufficient to prove that any
$M\in \mathscr{W}(\omega_{h_{\gamma}},\gamma\in C)$ is concentrated as a Radon measure on $\mathscr{N}'_{\mathfrak{R}}$.

\begin{thm}
  \label{thm:2}
  Let $(\omega_{h_{\gamma}})_{\gamma\in C}$ be a net of regular states on $\mathbf{W}_{h_{\gamma}}(V',\sigma)$ such that
  \begin{equation*}
    \sup_{\gamma\in C}\omega_{h_{\gamma}}\bigl(W_{h_{\gamma}}(0)\bigr)<\infty\; .
  \end{equation*}
  If there exists a nuclear barrelled space $\mathscr{N}$ such that $\mathscr{N}\subseteq V'$, and
  $(V/\Phi)_{\Phi\in F(\mathscr{N}'_{\mathfrak{R}})}$ is identified with a projective subfamily of
  $(V/\Phi)_{\Phi\in F(V)}$, then equicontinuity of the family $\mathbf{G}_B$ defined in~\eqref{eq:15} implies
  that any $M\in \mathscr{W}(\omega_{h_{\gamma}},\gamma\in C)$ is concentrated as a Borel Radon measure $\mu$ on
  $\mathscr{N}'_{\mathfrak{R}}$.
\end{thm}

The formulation of Theorem~\ref{thm:2} seems quite natural, however it may be difficult to apply it in
concrete situations, for the equicontinuity of $\mathbf{G}_B$ may be hard to prove without a detailed
knowledge of the generating functional.

\subsection{Duals of Banach spaces.}
\label{sec:locally-conv-spac}

In this section we consider concentration in dual spaces $\mathcal{X}'$, where $\mathcal{X}$ is a Banach
space. There is no loss of generality in considering $\mathcal{X}$ to be real, and all the results extend to
complex Banach spaces. It should also be possible to formulate a similar result for duals of locally convex
spaces, using the Bourbaki-Alaoglu variant of the Banach-Alaoglu's Theorem.


The idea is to adapt the procedure used in Section~\ref{sec:separ-hilb-spac} to prove tightness in
$\mathcal{X}'$ of a suitable projective subfamily of $M\in \mathscr{W}(\psi_{h_{\gamma}},\gamma\in
C)$. However, we will only be able to prove the concentration of $M$ as a Borel Radon measure $\mu$ on
$\mathcal{X}'$ endowed with the ultraweak dual topology. To avoid confusion, let us denote by
$\mathcal{X}_{\mathrm{w}}'$ the space $\mathcal{X}'$ endowed with the $\sigma(\mathcal{X}',\mathcal{X})$
topology. Let us recall that $(\mathcal{X}'_{\mathrm{w}})'=\mathcal{X}$.


\begin{lemma}[Tightness]
  \label{lemma:2}
  A cylindrical measure $(\mu_{\Psi})_{\Psi\in F(B_{\mathrm{w}}')}$ is concentrated as a Borel Radon measure
  $\mu$ on $\mathcal{X}_{\mathrm{w}}'$ if 
  for any $\varepsilon>0$, there exists a $\delta>0$ such that for any $\Psi\in F(\mathcal{X}_{\mathrm{w}}')$ of codimension
  $2n$, $n\in \mathds{N}$,
  \begin{equation}
    \label{eq:16}
    \mu_{\Psi}\bigl(\mathcal{X}_{\mathrm{w}}'/\Psi\setminus p_{\Psi}(\{x\in \mathcal{X}',\lVert x  \rVert_{\mathcal{X}'}^2\leq \delta\})\bigr)< \varepsilon\; .
  \end{equation}
\end{lemma}
\begin{proof}
  By Banach-Alaoglu's theorem the ball $\{x\in \mathcal{X}',\lVert x \rVert_{\mathcal{X}'}^2\leq \delta\}$ is compact in
  $\mathcal{X}_{\mathrm{w}}'$. In addition, the projective family of measures indexed by spaces of even codimension is a
  projective subfamily of $(\mu_{\Psi})_{\Psi\in F(\mathcal{X}_{\mathrm{w}}')}$ with cofinal index set. By a theorem of
  Prokhorov \citep[][Theorem
  I.21]{MR0426084
  }, \eqref{eq:16} is then sufficient to prove that $\mu$ is concentrated as a Radon measure on
  $\mathcal{X}_{\mathrm{w}}'$.
\end{proof}

The sufficient condition on quantum states is similar to the one given for separable Hilbert spaces.
\begin{hyp}
  \label{hyp:4}
  Let $(\psi_{h_{\gamma}})_{\gamma\in C}$ be a net of vectors on regular GNS representations of $\mathbf{W}_{h_{\gamma}}(\mathcal{X},\sigma)$, with
  \begin{equation*}
    \sup_{\gamma\in C}\lVert \psi_{h_{\beta}}  \rVert_{H_{\omega_{h_{\gamma}}}}^{}<\infty\; .
  \end{equation*}
  There exists $\delta>0$ and $\mathcal{K}>0$ such that uniformly in $\gamma\in C$,
  \begin{equation*}
    \sup_{\underset {\dim \mathcal{X}'_{\mathrm{w}}/\Psi=2n}{\Psi\in F(\mathcal{X}'_{\mathrm{w}})}}  \langle \psi_{h_{\gamma}}  , N_{\omega_{h_{\gamma}},\Psi}^{\delta} \, \psi_{h_{\gamma}} \rangle_{}\leq \mathcal{K}\; .
  \end{equation*}
\end{hyp}

The strategy is the same as the one used to prove \eqref{eq:9} in Theorem~\ref{thm:1}. Let
$K=\{x\in \mathcal{X},\lVert x \rVert_{\mathcal{X}}^2\leq 1\}$. The key point is to identify, for any $\Psi$ with
even codimension, $K\cap \Psi^{\circ}$ with $p_{\Psi}(K^{\circ})^{\circ}$. 
It is well-known that
\begin{equation*}
  K^{\circ}=\{\xi\in \mathcal{X}', \lVert \xi \rVert_{\mathcal{X}'}^2\leq 1\}\; .
\end{equation*}
In addition, $K\cap \Psi^{\circ}$ is a closed bounded absorbing absolutely convex neighbourhood of zero, and therefore a
closed ball centered in zero (in the locally compact topology of the finite dimensional real vector space
$\Psi^{\circ}$). It is possible to make the identification
$\mathcal{X}'_{\mathrm{w}}/\Psi \overset {\jmath}{\cong} (\mathcal{X}'_{\mathrm{w}}/\Psi)'$ in a way such that
for any $\xi\in (\mathcal{X}'_{\mathrm{w}}/\Psi)'$ and $[x]\in \mathcal{X}'_{\mathrm{w}}/\Psi$
\begin{equation}
  \label{eq:19}
  \jmath^{-1}(\xi)\,\cdot\, [x]=\xi([x])=\xi\, \cdot \, \jmath([x])\; . 
\end{equation}

\begin{lemma}
  \label{lemma:3}
  Let $(\mathcal{X}'_{\mathrm{w}}/\Psi)'\overset {^{\mathrm{t}}p_{\Psi}}{\cong}\Psi^{\circ}$ be the homeomorphism that yields, for any
  $x\in \mathcal{X}'_{\mathrm{w}}$ and $\xi\in (\mathcal{X}'_{\mathrm{w}}/\Psi)'$,
  \begin{equation}
    \label{eq:21}
    \xi([x])= \, ^{\mathrm{t}}p_{\Psi}(\xi)(x)\; .
  \end{equation}
  Then
  \begin{equation*}
    K\cap \Psi^{\circ}= \,^{\mathrm{t}}p_{\Psi}\bigl( p_{\Psi}(K^{\circ})^{\circ} \bigr)\;.
  \end{equation*}
\end{lemma}
\begin{proof}
  By definition,
  \begin{equation*}
    p_{\Psi}(K^{\circ})^{\circ}=\Bigl\{\xi\in (\mathcal{X}'_{\mathrm{w}}/\Psi)'\,,\, \forall x\in K^{\circ}\;\lvert \xi([x])  \rvert_{}^{}\leq 1\Bigr\}\; .
  \end{equation*}
  Applying the isomorphism, it follows by \eqref{eq:21} that
  \begin{equation*}
    \begin{split}
      ^{\mathrm{t}}p_{\Psi}\bigl( p_{\Psi}(K^{\circ})^{\circ} \bigr)=\Bigl\{\,^{\mathrm{t}}p_{\Psi}(\xi)\in \, ^{\mathrm{t}}p_{\Psi}\bigl((\mathcal{X}'_{\mathrm{w}}/\Psi)'\bigr)\,,\, \forall x\in K\,\lvert \, ^{\mathrm{t}}p_{\Psi} (\xi)(x)  \rvert_{}^{}\leq 1\Bigr\}\\
      =\Bigl\{\zeta\in \Psi^{\circ}\,,\, \forall x\in K^{\circ}\,\lvert \zeta(x)  \rvert_{}^{}\leq 1\Bigr\}=K^{\circ \circ}\cap \Psi^{\circ}= K\cap \Psi^{\circ}\; .
    \end{split}
  \end{equation*}
  In the last equality we have used the bipolar theorem for locally convex spaces in duality \citep[see
  e.g.][II.49 Corollaire 3]{MR633754
  }.
\end{proof}
From Lemma~\ref{lemma:3} it follows that
$\jmath^{-1}\bigl(\, ^{\mathrm{t}}p_{\Psi}^{-1}(K\cap \Psi^{\circ})\bigr)$ is a closed ball centered in zero
of $\mathcal{X}'_{\mathrm{w}}/\Psi$. Hence the tightness property \eqref{eq:16} of Lemma~\ref{lemma:2} is
satisfied iff
\begin{equation}
  \label{eq:20}
  \lim_{r\to \infty}\mu_{\Psi^{\circ}}\bigl(\{\zeta\in \Psi^{\circ}\,,\, \lvert \zeta \rvert_{}^2\geq r\}\bigr)=\lim_{r\to \infty}\mu_{\Psi}\bigl(\{z\in \mathcal{X}'_{\mathrm{w}}/\Psi\,,\, \lvert z\rvert_{}^2\geq r\}\bigr) =0\; ,
\end{equation}
where $\lvert \,\cdot \, \rvert_{}^{}$ denotes the norm on finite dimensional vector spaces, and the measure
$\mu_{\Psi^{\circ}}=\, ^{\mathrm{t}}p_{\Psi}\, _{*}\, \jmath\, _{*}\, \mu_{\Psi}$ is the pushforward of
$\mu_{\Psi}$. Equation~\eqref{eq:20} is analogous to the equivalence among~\eqref{eq:22} and~\eqref{eq:9}, and
therefore it is true for an $M\in \mathscr{W}(\psi_{h_{\gamma}},\gamma\in C)$, provided there exist $\delta>0$ and $C>0$ such that
\begin{equation}
  \label{eq:23}
  \langle \psi_{h_{\gamma}}  , N_{\omega_{h_{\gamma}},\Psi}^{\delta} \, \psi_{h_{\gamma}} \rangle_{}\leq \mathcal{K}\; ,
\end{equation}
as explained in the proof of Theorem~\ref{thm:1}. Obviously~\eqref{eq:23} is implied, for any
$\Psi\in F(\mathcal{X}'_{\mathrm{w}})$ of even codimension, by Hypothesis \ref{hyp:4}. Therefore we have proved the
following result.

\begin{thm}
  \label{thm:3}
  Let $\mathcal{X}$ be a Banach space. For a cylindrical measure
  \begin{equation*}
    M\in \mathscr{W}(\psi_{h_{\gamma}},\gamma\in C)
  \end{equation*}
  to be concentrated as a Borel Radon measure $\mu$ on $\mathcal{X}'_{\mathrm{w}}$, it is sufficient that
  \ref{hyp:4} is satisfied.
\end{thm}
\begin{remark}
  \label{rem:1}
  We are not aware of explicit examples of vectors $(\psi_{h_{\gamma}})_{\gamma\in C}$ that satisfy \ref{hyp:4} when
  $\mathcal{X}$ is not a (separable) Hilbert space. On the contrary, there are many explicit examples of $V$, $X$, $t(\cdot ,\cdot )$,
  and $(\psi_{h_{\gamma}})_{\gamma\in C}$ such that Theorem~\ref{thm:1} holds. It would be interesting to find
  explicit examples where Theorem~\ref{thm:3} can be applied and $\mathcal{X}$ is \emph{e.g.} a separable reflexive
  Banach space.
\end{remark}
\begin{remark}
  \label{rem:2}
  If $\mathcal{X}=L^{\frac{p}{p-1}} (\mathds{R}^d )$, $1<p<\infty$, then $\lVert \,\cdot \, \rVert_p^{}$ is a Borel measurable
  function in $\mathcal{X}'_{\mathrm{w}}=L^p_{\mathrm{w}}$, where $L^p_{\mathrm{w}}$ is $L^p$ with the
  $\sigma(L^p,L^{\frac{p}{p-1}})$ topology.
\end{remark}

\subsection{Applications of Theorem \ref{thm:1}}
\label{sec:some-expl-exampl}

In this section, let us discuss some other concrete situations in which the concentration of Wigner measures
on ``interesting'' spaces plays an important role.



\subsubsection{Massless fields and homogeneous Sobolev spaces, revisited}
\label{sec:massl-fields-homog}
\begin{itemize}
  \setlength{\itemsep}{7pt}
\item $X=\mathscr{S}(\mathds{R}^d)$
\item $t(f,g )=\int_{\mathds{R}^d}^{}\lvert k \rvert_{}^{2s}\bar{f}(k)g(k) \mathrm{d}k$, $s<\frac{d}{2}$
\item $V=\mathbb{F}\, L^2(\mathds{R}^d)$, $\sigma(\,\cdot\, ,\,\cdot\, )=\Im \langle \,\cdot \, , \,\cdot \, \rangle_2$
\item $\mathbf{m}\bigl(f(k)\bigr)= \lvert k \rvert_{}^{s}f(k)$
\item $(\psi_h)_{h\in (0,1)}\subset \Gamma_{\mathrm{s}}\bigl(L^2 (\mathds{R}^d )\bigr)_h$ (Fock representation)
\item $\exists \delta>0$, $\exists \mathcal{K}>0$, $\langle \psi_h , \mathrm{d}\Gamma_h(\lvert k \rvert_{}^{2s})^{\delta}\psi_h \rangle_{\Gamma_{\mathrm{s}}}\leq \mathcal{K}$
\end{itemize}
\vspace{3pt} These definitions provide the link between the motivating problem of
Section~\ref{sec:motivation}, and Theorem~\ref{thm:1}. In fact, from the definitions it follows that
\begin{equation*}
  X_q=L^2 (\mathds{R}^d, \lvert k \rvert_{}^{2s}\mathrm{d}k )=\mathscr{F}\dot{H}^s(\mathds{R}^d)\; ;
\end{equation*}
and
\begin{equation*}
  \mathbf{n}_{h,\infty}^{\delta}(\cdot )=\langle \,\cdot\, , \mathrm{d}\Gamma_h(\lvert k \rvert_{}^{2s})^{\delta}\,\cdot \, \rangle_{\Gamma_{\mathrm{s}}}\; ,
\end{equation*}
where $\mathrm{d}\Gamma_h(\lvert k \rvert_{}^{2s})$ is the second quantization of the multiplication operator
$\lvert k \rvert_{}^{2s}$ on $L^2(\mathds{R}^d)$. Therefore by Theorem~\ref{thm:1},
\begin{align*}
  M\in \mathscr{W}\bigl(\psi_{h},\, h\in (0,1)\bigr)\, &\Rightarrow \, M \text{ is concentrated as } \mu \in \mathscr{M}_{\mathrm{rad}}\bigl(\mathscr{F}\dot{H}^s(\mathds{R}^d)\bigr) \; ;\\
  &\int_{\mathscr{F}\dot{H}^s}^{}\lVert f  \rVert_{\mathscr{F}\dot{H}^s}^{\delta}  \mathrm{d}\mu(f)\leq \mathcal{K}\; .
\end{align*}
The above result, with $s=\frac{1}{2}$ and $d\geq 2$, can be used to solve the question of convergence of the
ground state energy in both the classical and quasi-classical limit, for massless Nelson-type models; these
questions were left open respectively
in~\citep{Ammari:2014aa
} and~\citep{Correggi:2017aa
}.

If in addition, $(\psi_h)_{h\in (0,1)}$ satisfies also
\begin{equation*}
  \langle \psi_h  , \mathrm{d}\Gamma_h(1)^{\tilde{\delta}}\psi_h \rangle_{\Gamma_{\mathrm{s}}}\leq \tilde{\mathcal{K}}\; ,
\end{equation*}
then any $M\in \mathscr{W}\bigl(\psi_{h},\, h\in (0,1)\bigr)$ is concentrated as a Radon measure $\mu$ that is a Borel
measure on both $L^2(\mathds{R}^d)$ and $\mathscr{F} \dot{H}^s(\mathds{R}^d)$. Furthermore, the measure of the
symmetric difference is zero:
\begin{equation*}
  \mu\bigl(L^2 (\mathds{R}^d  )\ominus \mathscr{F}\dot{H}^s(\mathds{R}^d) \bigr)=0\; .
\end{equation*}

Finally, let us remark that -- for time-evolved states in massless dynamical quantum theories (such as the
Pauli-Fierz model) -- the condition
\begin{equation*}
  \langle \psi_h(t) , \mathrm{d}\Gamma_h(\lvert k \rvert_{})\psi_h(t) \rangle_{\Gamma_{\mathrm{s}}}\leq \mathcal{K}(t)
\end{equation*}
is much easier to prove, provided it holds at time zero, than
\begin{equation*}
  \langle \psi_h(t) , \mathrm{d}\Gamma_h(1) \psi_h(t) \rangle_{\Gamma_{\mathrm{s}}}\leq \tilde{\mathcal{K}}(t)\; .
\end{equation*}

\subsubsection{Non-homogeneous Sobolev spaces with negative index}
\label{sec:non-homog-sobol}
\begin{itemize}
  \setlength{\itemsep}{7pt}
\item $X=L^2(\mathds{R}^d)$
\item $t(f,g )=\int_{\mathds{R}^d}^{}(1-\Delta)^{2r}\bar{f}(x)g(x) \mathrm{d}x$, $r<0$
\item $V=\mathbb{F}\, L^2(\mathds{R}^d)$, $\sigma(\,\cdot\, ,\,\cdot\, )=\Im \langle \,\cdot \, , \,\cdot \, \rangle_2$
\item $\mathbf{m}\bigl(f(x)\bigr)= (1-\Delta)^{r} f(x)$
\item $(\psi_h)_{h\in (0,1)}\subset \Gamma_{\mathrm{s}}\bigl(L^2 (\mathds{R}^d )\bigr)_h$
\item $\exists \delta>0$, $\exists \mathcal{K}>0$, $\langle \psi_h , \mathrm{d}\Gamma_h\bigl((1-\Delta)^{2r}\bigr)^{\delta}\psi_h \rangle_{\Gamma_{\mathrm{s}}}\leq \mathcal{K}$
\end{itemize}
In this case, $X_q=H^{r}(\mathds{R}^d)\supset L^2 (\mathds{R}^d )$, the non-homogeneous Sobolev space. In studying
the thermodynamic properties of bosonic quantum fields, and their classical counterparts, one has to deal with
quantum states whose corresponding Wigner measures are concentrated outside of the classical phase space. The
Gibbs states provide a very interesting example
\citep{Lewin:2014aa
,Frohlich:2016aa
}.

By means of Theorem~\ref{thm:1}, we know that
\begin{align*}
  M\in \mathscr{W}\bigl(\psi_{h},\, h\in (0,1)\bigr)\, &\Rightarrow \, M \text{ is concentrated as } \mu \in \mathscr{M}_{\mathrm{rad}}\bigl(H^{r}(\mathds{R}^d)\bigr) \; ;\\
  &\int_{H^{r}}^{}\lVert f  \rVert_{H^{r}}^{\delta}  \mathrm{d}\mu(f)\leq \mathcal{K}\; .
\end{align*}
Let us provide an explicit example of a sequence of vectors in the Fock space whose corresponding Wigner
measure is concentrated on $H^r\setminus L^2$. Let
$\Omega\in \Gamma_{\mathrm{s}}\bigl(L^2 (\mathds{R}^d )\bigr)$ be the Fock vacuum vector; and consider a
family $(\psi_h)_{h\in (0,1)}$ of squeezed coherent vectors of the form
\begin{equation*}
  \psi_{h}=e^{\frac{1}{h}(a^{*}(f_h)-a(f_h))}\,\Omega\; ,
\end{equation*}
$f_h\in L^2 (\mathds{R}^d )$. Then the associated generating functional takes the form
\begin{equation*}
  G_{\psi_h}(\cdot )=e^{i\, \Re\langle \,\cdot\,   , f_h \rangle_2}e^{-\frac{h}{2} \lVert f_h  \rVert_2^2}\; ,
\end{equation*}
and it satisfies
\begin{equation*}
  \langle \psi_h , \mathrm{d}\Gamma_h\bigl((1-\Delta)^{2r}\bigr)\psi_h \rangle_{\Gamma_{\mathrm{s}}}=\lVert f_h  \rVert_{H^r}^2\; ,\; \langle \psi_h , \mathrm{d}\Gamma_h(1)\psi_h \rangle_{\Gamma_{\mathrm{s}}}=\lVert f_h  \rVert_2^2\; .
\end{equation*}
If the sequence $(f_h)_{h\in (0,1)}$ satisfies:
\begin{equation*}
  f_h\to_{H^r}f\in H^r\setminus L^2 \; ,\; \sup_{h\in (0,1)}\lVert f_h  \rVert_{H^r}^{}=\mathcal{K}<\infty\; ,\; 
\end{equation*}
then any $M\in \mathscr{W}\bigl(\psi_h,h\in (0,1)\bigr)$ is concentrated as a measure $\mu$ on $H^r$, but we may have
$\mu(L^2)=0$. In fact, if
\begin{equation*}
  \lVert f_h \rVert_2^2=O(h^{\varepsilon-1})\; ,\; 0<\varepsilon< 1\; ;
\end{equation*}
then $\mathscr{W}\bigl(\psi_h,h\in (0,1)\bigr)=\{M_{\psi}\}$, where $M_{\psi}$ concentrates as the atomic measure
$\delta_f$ on $f\in H^r\setminus L^2$: for any $\varphi\in H^{\lvert r \rvert_{}^{}}$
\begin{equation*}
  \lim_{h\to 0}G_{\psi_h}(\varphi)=e^{i\, \Re  \int_{\mathds{R}^d}^{}\bar{\varphi}(x)f(x)  \mathrm{d}x}=\int_{H^r}^{}e^{i\, \Re \int_{\mathds{R}^d}^{}\bar{\varphi}(x)g(x)  \mathrm{d}x}  \mathrm{d}\,\delta_f(g)\; .
\end{equation*}

\subsubsection{Free quantum evolution and pushforward measure}
\label{sec:free-evol-pushf}

Recover the setting of Section~\ref{sec:massl-fields-homog}. Let us now discuss how a (free) dynamical group
acting on quantum states transforms the corresponding Wigner measures accordingly, following the classical
dynamics. Let $\omega(k)$ be any self-adjoint multiplication operator on $L^2(\mathds{R}^d , \mathrm{d}k)$,
and $(e^{it \omega(k)})_{t\in \mathds{R}}$ the associated strongly continuous unitary group. Then
$(e^{it \omega(k)})_{t\in \mathds{R}}$ extends to a strongly continuous unitary group on
$L^2 (\mathds{R}^d, \lvert k \rvert_{}^{2s}\mathrm{d}k )$. The associated second quantized unitary operator
$\Gamma_h(e^{it\omega})=e^{it \mathrm{d}\Gamma_h(\omega)}$ commutes with
$\mathrm{d}\Gamma_h(\lvert k \rvert_{}^{2s})$, and therefore
$\langle \psi_h , \mathrm{d}\Gamma_h(\lvert k \rvert_{}^{2s})^{\delta}\psi_h \rangle_{\Gamma_{\mathrm{s}}}\leq \mathcal{K}$ iff
\begin{equation*}
  \forall t\in \mathds{R}\;,\; \langle \psi_h , e^{-it \mathrm{d}\Gamma_h(\omega)}\mathrm{d}\Gamma_h(\lvert k \rvert_{}^{2s})^{\delta}e^{-it \mathrm{d}\Gamma_h(\omega)}\psi_h \rangle_{\Gamma_{\mathrm{s}}}\leq \mathcal{K}\; .
\end{equation*}
It follows that any $M_t\in \mathscr{W}\bigl(e^{i t \mathrm{d}\Gamma_h(\omega)}\psi_h,h\in (0,1)\bigr)$ is concentrated as a
Radon measure $\mu_t$ on $\mathscr{F}\dot{H}^s(\mathds{R}^d)$, for any $t\in \mathds{R}$. As showed in
\citep[][Proposition 6.1]{Falconi:2016ab
}, $M_t= e^{it\omega}\, _{*}\, M_0$, the pushforward image by $e^{it\omega}\in \mathcal{B}(L^2)$ of the cylindrical measure
$M_0$. Now since $e^{it\omega}\in \mathcal{B}(\mathscr{F} \dot{H}^s)$ as well, it follows that
\begin{equation*}
  \mu_t= e^{it\omega}\, _{*}\, \mu_0\; .
\end{equation*}

{\footnotesize
\bibliography{/Users/lobo/Library/texmf/bibtex/bib}}
\end{document}